\documentclass[12pt]{amsart}
\usepackage[top=0.9in, bottom=1in, left=1in, right=1in]{geometry}
\usepackage{amsmath, amssymb, amsthm}
\usepackage{verbatim}
\usepackage{graphicx}
\usepackage{enumerate}
\usepackage{mathrsfs}

\newtheorem{thm}{Theorem}[section]
\newtheorem{prop}[thm]{Proposition}
\newtheorem{lemma}[thm]{Lemma}
\newtheorem{cor}[thm]{Corollary}

\def\XXint#1#2#3{{\setbox0=\hbox{$#1{#2#3}{\int}$ }
\vcenter{\hbox{$#2#3$ }}\kern-.6\wd0}}

\theoremstyle{definition}
\newtheorem{definition}[thm]{Definition}
\newtheorem{example}[thm]{Example}

\theoremstyle{remark}

\newtheorem{remark}[thm]{Remark}

\numberwithin{equation}{section}

\newtheoremstyle{ser}
{8pt}
{8pt}
{\it}
{}
{\sf}
{:}
{6mm}
{}

\newtheoremstyle{serr}
{8pt}
{8pt}
{\normalfont}
{}
{\sf}
{.}
{6mm}
{}

\theoremstyle{ser}

\theoremstyle{serr}

\begin{document}
\title{Jensen's Inequality in Finite Subdiagonal Algebras}

\author{Soumyashant Nayak}
\address{Smilow Center for Translational Research, University of Pennsylvania, Philadelphia, PA 19104}
\email{nsoum@upenn.edu}
\urladdr{https://nsoum.github.io/} 


\begin{abstract}
Let $\mathscr{M}$ be a finite von Neumann algebra with a faithful normal tracial state $\tau$ and $\mathfrak{A}$ be a finite subdiagonal subalgebra of $\mathscr{M}$ with respect to a $\tau$-preserving faithful normal conditional expectation $\Phi$ on $\mathscr{M}$. Let $\Delta$ denote the Fuglede-Kadison determinant corresponding to $\tau$. For $X \in \mathscr{M}$, define $|X| := (X^*X)^{\frac{1}{2}}$. In 2005, Labuschagne proved the so-called Jensen's inequality for finite subdiagonal algebras {\it i.e.}\ $\Delta(\Phi(A)) \le \Delta(A)$ for an operator $A \in \mathfrak{A}$, thus resolving a long-standing open problem posed by Arveson in 1967. In this article, we prove the following more general result: $\tau(f( |\Phi(A)|)) \le \tau(f( |A|))$ for $A \in \mathfrak{A}$ and any increasing continuous function $f : [0, \infty) \to \mathbb{R}$ such that $f \circ \exp$ is convex on $\mathbb{R}$. Under the additional hypotheses that $A$ is invertible in $\mathscr{M}$ and $f \circ \exp$ is strictly convex, we have $\tau(f( |\Phi(A)|)) = \tau(f( |A|)) \Longleftrightarrow \Phi(A) = A$. As an application, we show that for $A \in \mathfrak{A}$ the point spectrum of $A$ is contained in the point spectrum of $\Phi(A)$, though such a conclusion does not hold in general for their spectra.

\vskip 0.01in \noindent
{\bf MSC2010 subject classification:} 47C15, 46L10, 15A45
\end{abstract}

\maketitle

\section{Introduction}

Let $\mathscr{M}$ be a finite von Neumann algebra with a faithful normal tracial state $\tau$. A conditional expectation from $\mathscr{M}$ onto a von Neumann subalgebra $\mathscr{N}$ is defined to be a positive linear map $\Phi : \mathscr{M} \to \mathscr{N}$ which preserves the identity operator $I$ and satisfies $\Phi(YX) = Y \Phi(X)$ for all $X \in \mathscr{M}$ and $Y \in \mathscr{N}$. If $\tau(\Phi(X)) = X$ for all $X \in \mathscr{M}$, we say that $\Phi$ is $\tau$-preserving. A $\tau$-preserving conditional expectation is automatically faithful and normal. An archetypal example of a conditional expectation is the map from $M_n(\mathbb{C})$, the set of $n \times n$ complex matrices,  to $D_n(\mathbb{C})$, the set of diagonal matrices, which sends a matrix to a diagonal matrix with the same diagonal entries. With this in mind, we may think of $\mathscr{N}$ as the `diagonal' subalgebra of $\mathscr{M}$ with respect to $\Phi$.

\begin{definition}
Let $\Phi$ be a faithful normal conditional expectation from $\mathscr{M}$ onto $\mathscr{N}$. Let $\mathfrak{A}$ be a ultraweakly closed subalgebra of $\mathscr{M}$ containing the identity operator $I$ such that $\mathscr{N} = \mathfrak{A} \cap \mathfrak{A}^*$ (the diagonal of $\mathfrak{A}$). Then $\mathfrak{A}$ is said to be a finite subdiagonal subalgebra of $\mathscr{M}$ with respect to $\Phi$ if:
\begin{itemize}
\item[(i)] $\mathfrak{A} + \mathfrak{A}^*$ is ultraweakly dense in $\mathscr{M}$,
\item[(ii)] $\Phi(AB) = \Phi(A) \Phi(B)$ for all $A, B \in \mathfrak{A}$,
\item[(iii)] $\Phi$ is $\tau$-preserving for a faithful normal tracial state $\tau$ on $\mathscr{M}$.
\end{itemize}
\end{definition}
If $\mathfrak{A}$ is not properly contained in another finite subdiagonal algebra with respect to $\Phi$, then $\mathfrak{A}$ is said to be {\it maximal}. By \cite[Theorem 7]{exel}, (ultraweakly closed) finite subdiagonal algebras are automatically maximal. Because of the maximality, we have at our disposal the Arveson-Beurling factorization theorem (as paraphrased in Lemma \ref{lem:arveson_fac}).

In \cite{arveson_analyticity}, Arveson developed the theory of subdiagonal algebras with a view towards a unified treatment for various results known then concerning non-self-adjoint operator algebras. We note some examples of finite subdiagonal algebras below to illustrate the scope of the results in this article.

\begin{example}
\label{ex:upper_tri}
Consider $\mathscr{M} = M_n(\mathbb{C})$ equipped with the trace $\tau(X) = \frac{1}{n} (\sum_{i=1}^n X_{ii})$.
\begin{itemize}
\item[(i)] Let $\mathscr{N} = D_n(\mathbb{C})$ and $\Phi$ be the previously described diagonal map onto $D_n(\mathbb{C})$. Then the algebra of upper triangular matrices is a finite subdiagonal algebra with respect to $\Phi$.
\item[(ii)] Let $n_1, \cdots, n_k \in \mathbb{N}$ such that $\sum_{i=1}^k n_i = n$. Embed $M_{n_1}(\mathbb{C}) \oplus \cdots \oplus M_{n_k}(\mathbb{C})$ into $M_n(\mathbb{C})$ as principal diagonal blocks of a $n \times n$ matrix with other entries $0$. In this sense, let $\mathscr{N} = M_{n_1}(\mathbb{C}) \oplus \cdots \oplus M_{n_k}(\mathbb{C})$ and $\Phi : \mathscr{M} \to \mathscr{N}$ be the natural projection. Then the algebra of block upper triangular matrices with respect to $\mathscr{N}$ is a finite subdiagonal algebra with respect to $\Phi$.
\end{itemize}
\end{example}

\begin{example}
\label{ex:hardy}
Let $(\mathbb{T}, \mu)$ denote the unit circle in $\mathbb{C}$ with the uniform probability measure. The Hardy space $H^{\infty}(\mathbb{T}) \subset L^{\infty}(\mathbb{T})$ consists of essentially bounded functions on $\mathbb{T}$ with vanishing (strictly) negative Fourier coefficients. For $f \in L^{\infty}(\mathbb{T})$, let $\Phi(f) = (\int_{\mathbb{T}} f \; \mathrm{d}\mu) I $ and $\tau(f) = \int_{\mathbb{T}} f \; \mathrm{d}\mu$. Note that $\Phi$ is a $\tau$-preserving conditional expectation.
\begin{itemize}
\item[(i)] Then $H^{\infty}(\mathbb{T})$ is a finite subdiagonal subalgebra of $L^{\infty}(\mathbb{T})$ with respect to $\Phi$.
\item[(ii)] Consider $\mathscr{M} = M_n(L^{\infty}(\mathbb{T})) \cong L^{\infty}(\mathbb{T}) \otimes M_n(\mathbb{C})$ and $\mathscr{N} = M_n(H^{\infty}(\mathbb{T})) \cong H^{\infty}(\mathbb{T}) \otimes M_n(\mathbb{C})$ and the natural trace obtained from the tensor product of $\tau$ with the normalized trace on $M_n(\mathbb{C})$. Let $\Psi = \Phi \otimes I$ be the tensor product of $\Phi$ with the identity map on $M_n(\mathbb{C})$. Then $M_n(H^{\infty}(\mathbb{T}))$ is a finite subdiagonal subalgebra of $M_n(L^{\infty}(\mathbb{T}))$ with respect to $\Psi$.
\end{itemize}
\end{example}

An important goal in \cite{arveson_analyticity} was to transplant, apart from the Beurling factorization theorem, Jensen's inequality in $H^{\infty}(\mathbb{T})$ to the more general setting of finite subdiagonal algebras. For a bounded analytic function $f$ on the open unit disc $\mathbb{D}$ in $\mathbb{C}$ ($f \in H^{\infty}(\mathbb{D})$) and $z \in \mathbb{D}$, a version of Jensen's inequality states that $$|f(z)| \le \exp \Big( \frac{1}{2 \pi} \int_{0}^{2\pi} (\log | f(e^{i \theta})|) P_z(\theta) \; d\theta \Big),$$
where $P_z(\cdot)$ denotes the Poisson kernel representing evaluation at point $z$.\ Let $\mathfrak{A}$ denote a finite subdiagonal subalgebra of $\mathscr{M}$ with respect to a $\tau$-preserving conditional expectation $\Phi$. Let $\Delta$ denote the analytic extension of the Fuglede-Kadison determinant (cf.\ \cite{fuglede-kadison}) associated with $\tau$. In \cite[\S 4.4]{arveson_analyticity}, Arveson proposed the following generalization of the above mentioned Jensen's inequality: $$\Delta (\Phi(A)) \le \Delta(A), \textrm{ for every } A \in \mathfrak{A}.$$
In \cite[Theorem 4.4.3]{arveson_analyticity}, he further proved the equivalence of the above inequality to a noncommutative version of a classical result by Szeg\H{o}. Although Arveson proved the inequality in \cite[\S 5]{arveson_analyticity} for many important examples, he described the problem of ascertaining its validity for {\it all} finite subdiagonal algebras as  ``the most exasperating of the open questions about finite subdiagonal algebras ...''(cf.\ \cite[Question 4.4.1]{arveson_analyticity}). After nearly four decades, in 2005, Labuschagne resolved the problem in the affirmative (cf.\ \cite[Theorem 3]{labuschagne}).

For $X \in \mathscr{M}$, let $|X| := (X^*X)^{\frac{1}{2}}$. In this paper, we suitably adapt Labuschagne's strategy and prove a majorization inequality (cf.\ Theorem \ref{thm:jensen_subdiagonal}) which leads to the following stronger result (cf.\ \S $5$).
\vskip 0.07in

\noindent {\bf Theorem 5.1} (Jensen's inequality). \textsl{Let $f$ be an increasing continuous function on $[0, \infty)$ such that $f \circ \exp$ is convex. For $A \in \mathfrak{A}$, we have
$$\tau(f(|\Phi(A)|)) \le \tau(f(|A|)).$$
In addition, if $A$ is invertible in $\mathscr{M}$ and $f \circ \exp$ is strictly convex, equality holds if and only if $\Phi(A) = A$.
}
\vskip 0.07in

We briefly discuss the significance of the above version of Jensen's inequality in matrix analysis and statistics when considered in the context of (block) upper triangular matrices. Matrix factorization results play an important role in efficiently solving the normal equations arising in linear regression to determine the ordinary least squares solution. Two commonly used algorithms are the Cholesky decomposition which decomposes any positive-definite Hermitian matrix as $T^*T$ where $T$ is an invertible upper triangular matrix, and the $QR$ decomposition which decomposes a real (complex, respectively) square matrix into the product of an orthogonal (unitary, respectively) matrix and an upper triangular matrix. Determinant inequalities involving positive-definite matrices (such as the Hadamard-Fischer inequality) and upper triangular matrices are an active area of research in matrix analysis and statistics because of their ubiquity in stability/error estimates. For example, in \cite[Theorem 1]{drury}, a lower bound for the $k^{\textrm{th}}$ compound condition number of a  positive-definite $2 \times 2$ block matrix is obtained in terms of the canonical correlations of the block matrix. One of the lemmas used in the proof is the following determinant inequality for upper triangular matrices.

\begin{lemma}[Drury; {\cite[Lemma 4]{drury}}]
\label{lem:drury}
\textsl{
For an $n \times n$ complex upper triangular matrix $T$, we have $$\det(I_n + T^*T) \ge \prod_{i=1}^n (1 + |t_{ii}|^2),$$
with equality if and only if $T$ is diagonal.
}
\end{lemma}

The reader may consult \cite{can_corr} for a discussion of the relationship of canonical correlations to the inefficiency of the ordinary least squares method when the error terms in a Gauss-Markov linear model may be correlated. A general form of the preceding inequality is discussed in \cite[Theorem 3]{lin} which we state below.

\begin{thm}[Lin; {\cite[Theorem 3]{lin}}]
\label{thm:lin}
\textsl{
Let $T = \begin{bmatrix}
X & Y\\
0 & Z
\end{bmatrix}$ be an $n \times n$ complex matrix, where $X, Z$ are $k \times k, (n-k) \times (n-k)$ blocks, respectively. Then for any $r > 0$,$$\det(I_n + |T|^r) \ge \det (I_k + |X|^r) \cdot \det (I_{n-k} + |Z|^r),$$
with equality if and only if $Y = 0$.
}
\end{thm}

By suitably choosing $f$ in the context of Example \ref{ex:upper_tri}, Theorem \ref{thm:main_ineq} captures the above two determinant inequalities for (block) upper triangular matrices (cf.\ Corollary \ref{cor:app},(ii)). In \cite{nayak_hadamard}, the author discusses determinant inequalities in finite von Neumann algebras involving positive operators and their corresponding diagonals. This article may be considered as a counterpart involving ``subdiagonal operators" and their corresponding diagonals.

Finally as an application of Corollary \ref{cor:app}, (i), we show that for $A \in \mathfrak{A}$ the point spectrum of $A$ is contained in the point spectrum of $\Phi(A)$ (cf.\ Theorem \ref{thm:point_spec}, (ii)), though such a conclusion does not hold in general for their spectra.

\subsection*{Acknowledgments}
I would like to thank Minghua Lin for sharing his results on determinant inequalities involving block upper triangular matrices which primed me towards exploring whether similar results hold in finite subdiagonal algebras.\ I would also like to express my gratitude towards Louis Labuschagne for helpful e-mail correspondence regarding the problem discussed in this article. Lastly I am grateful for the comments of an anonymous referee which helped improve the presentation.

\section{Preliminaries}

In this section, we set up the notation used in the article and discuss some basic results on generalized $s$-numbers which are important for establishing the main results in this article.
\subsection{Notation}
Throughout this article, $\mathscr{M}$ denotes a finite von Neumann algebra with a faithful normal tracial state $\tau$. The identity operator of $\mathscr{M}$ is denoted by $I$. We consider a $\tau$-preserving faithful normal conditional expectation $\Phi$ from $\mathscr{M}$ onto a von Neumann subalgebra $\mathscr{N}$, and a finite subdiagonal subalgebra $\mathfrak{A}$ of $\mathscr{M}$ with respect to $\Phi$. The set of inverses of invertible operators in $\mathfrak{A}$ is denoted by $\mathfrak{A}^{-1}.$ Thus $\mathfrak{A} \cap \mathfrak{A}^{-1}$ contains those invertible operators in $\mathfrak{A}$ whose inverse also lies in $\mathfrak{A}$. 

For an operator $X \in \mathscr{M}$, we define $|X| := (X^*X)^{\frac{1}{2}}$. We denote its spectrum by $\sigma(X)$, its spectral radius by $r(X)$ and its point spectrum (set of eigenvalues) by $\sigma_p(X)$. The projection onto the closure of the range of $X$ {\it i.e.}\ the range projection of $X$, is denoted by $R(X)$. The projection onto the nullspace of $X$ is denoted by $N(X)$. We generally use $E$ to denote projections in $\mathscr{M}$, and $H$ to denote positive operators in $\mathscr{M}$. The set of non-negative real numbers is denoted by $\mathbb{R}_{+}$.

\subsection{Generalized $s$-numbers}

For $X \in \mathscr{M}$, the $t^{\textrm{th}}$ generalized $s$-number is defined as $$\mu_t(X) := \inf \{ \| XE \| : E \textrm{ is a projection in } \mathscr{M}, \tau(I-E) \le t\},\; \textrm{ for } t \ge 0.$$

As $\tau(I) = 1$, note that $\mu_t(X) = 0$ for $t > 1$. Also it is clear that $\mu_0(X) = \|X\|$. We paraphrase some pertinent results about generalized $s$-numbers from \cite{fack} below. The reader may also refer to the exposition in \cite[\S 2]{fack-kosaki}.

\begin{lemma}
\label{lem:s_number}
\textsl{
For $X, Y \in \mathscr{M}$ and $t \in [0, 1]$ and a continuous increasing function $f : \mathbb{R}_{+} \to \mathbb{R}_{+}$, we have:
\begin{itemize}
\item[(i)] $\mu_t(X) \le \|X\|$, and $|\mu_t(X) - \mu_t(Y)| \le \|X-Y\|$,
\item[(ii)] The map $s \in [0, 1] \mapsto \mu_s(X)$ is decreasing and right-continuous,
\item[(iii)] $\mu_t(X) = \mu_t(|X|) = \mu_t(X^*)$ and $\mu_t(\alpha X) = |\alpha| \mu_t(X)$ for $\alpha \in \mathbb{C}$,
\item[(iv)] $\mu_t(X) \le \mu_t(Y)$ if $0 \le X \le Y$,
\item[(v)] $\mu_t(f(|X|)) = f(\mu_t(|X|))$,
\item[(vi)] $\tau(f(|X|)) = \int_{0}^1 f(\mu_s(X)) \; \textrm{d}s$.
\end{itemize}
}
\end{lemma}

\begin{remark}
\label{rmrk:log_tr}
Let $H$ be a positive operator in $\mathscr{M}$ and $\lambda > 0$. Clearly $\mu_t(\lambda I + H) = \lambda + \mu_t(H)$ for $t \in [0, 1]$. Note that the map $t \in \mathbb{R}_{+} \mapsto \log(\lambda + t)$ is an increasing continuous function. In Lemma \ref{lem:s_number},(vi), the restrictive hypothesis that $f$ be positive-valued is not necessary with the trace in view as one may consider the function $f - f(0)$ and use the fact that $\tau(I) = 1 = \int_{0}^1 \; \textrm{d}s$. Thus $$\tau(\log (\lambda I + H)) = \int_{0}^1 (\log (\mu_s(\lambda I + H)) \; \textrm{d}s,$$
and for an invertible positive operator $H \in \mathscr{M}$, we have $$\tau(\log H) = \int_{0}^1 \log \mu_s(H) \; \textrm{d}s.$$
\end{remark}

\begin{remark}
\label{rmrk:variation}
For $X \in \mathscr{M}$, let $\mathscr{A}$ be any von Neumann subalgebra of $\mathscr{M}$ containing $|X|$. By \cite[Remark 2.3.1]{fack-kosaki}, note that $$\mu_t(X) = \inf \{ \|XE\| :  E \textrm{ is a projection in } \mathscr{A}, \tau(I - E) \le t\},\; \textrm{ for } t \ge 0.$$
\end{remark}

For an operator $X$ in $\mathscr{M}$ and $t \in (0,1]$, we define $$\Sigma_t(X) := \int_{0}^t \mu_s(X) \; \textrm{d}s.$$
Note that $\Sigma_1(X) = \tau(|X|)$. 

\begin{lemma}
\label{lem:cum_sum}
\textsl{
Let $X, Y \in \mathscr{M}$ and $r > 0$. For $t \in [0, 1]$, we have 
\begin{itemize}
\item[(i)] $\Sigma_t(X + Y) \le \Sigma_t(X) + \Sigma_t(Y)$, 
\item[(ii)] $\Sigma_t(|XY|^r) = \Sigma_t(\big| |X| |Y^*| \big|^r)$,
\item[(iii)] $\Sigma_t(|XY|^{\frac{r}{2}}) \le \frac{1}{2} \big( \Sigma_t(|X|^r) + \Sigma_t(|Y|^r) \big)$.
\end{itemize}
}
\end{lemma}
\begin{proof}
\begin{itemize}
\item[(i)] Restatement of \cite[Theorem 3.2]{fack}.
\item[(ii)] By Lemma \ref{lem:s_number},(iii), we get $\mu_t(|XY|^2) = \mu_t(Y^*|X|^2Y) = \mu_t(\big||X|Y\big|^2) = \mu_t(\big|Y^*|X|\big|^2)$. Reusing the argument for $Y^*, |X|$ we have  $\mu_t(\big|Y^*|X|\big|^2) = \mu_t(\big||X||Y^*|\big|^2)$. Thus for $s \in [0, 1]$, we have $\mu_s(|XY|^2) = \mu_s(\big||X||Y^*|\big|^2) \Rightarrow \mu_s(|XY|^2)^{\frac{r}{2}} = \mu_s(\big||X||Y^*|\big|^2)^{\frac{r}{2}} \Rightarrow \mu_s(|XY|^r) = \mu_s(\big||X||Y^*|\big|^r)$. Taking integrals with respect to $s$, we get the desired equality.
\item[(iii)] 
\begin{align*}
&\phantom{=} \frac{1}{2} \big( \Sigma_t(|X|^r) + \Sigma_t(|Y|^r) \big) &\\
&= \frac{1}{2} \Big( \int_{0}^t \mu_s(|X|^r) \; \textrm{d}s + \int_{0}^t \mu_s(|Y|^r) \; \textrm{d}s \Big) & \\
&= \frac{1}{2} \Big( \int_{0}^t \mu_s(X)^r \; \textrm{d}s + \int_{0}^t \mu_s(Y)^r \; \textrm{d}s \Big) & (\textrm{Lemma \ref{lem:s_number},(v)})\\
&\ge \int_{0}^t (\mu_s(X)\mu_s(Y))^{\frac{r}{2}} \; \textrm{d}s & (\textrm{AM-GM inequality}) \\
&\ge \int_{0}^t \mu_s(XY)^{\frac{r}{2}} \; \textrm{d}s & (\textrm{\cite[Theorem 4.3(ii)]{fack} for }t \mapsto t^{\frac{r}{2}})\\
&= \Sigma_t(|XY|^{\frac{r}{2}}). &
\end{align*}
\end{itemize}
\end{proof}

\begin{lemma}
\label{lem:mon_conv}
\textsl{
Let $(X_n)_{n \in \mathbb{N}} \subset \mathscr{M}$ be a sequence of operators converging uniformly to $X \in \mathscr{M}$. For $t \in [0, 1]$, we have $$\lim_{n \rightarrow \infty} \Sigma_t(X_n) = \Sigma_t(X).$$
}
\end{lemma}
\begin{proof}
For $n \in \mathbb{N}, t \in [0, 1]$, using Lemma \ref{lem:s_number},(i), we have $$|\Sigma_t(X_n) - \Sigma_t(X)| \le \int_{0}^t |\mu_s(X_n) - \mu_s(X)| \; \textrm{d}s \le t\|X_n - X\| \le \|X_n - X\|.$$
Taking the limit as $n \rightarrow \infty$, we get the desired result.
\end{proof}

\section{A Collection of Useful Lemmas}
In this section, we collect some results that are useful in our discussion in \S \ref{sec:main}, \S \ref{sec:app}. We state some of them without proof citing the appropriate reference in the literature.   

\begin{lemma}[generalized Schwarz inequality]
\label{lem:schwarz_ineq}
\textsl{For an operator $X \in \mathscr{M}$, we have $|\Phi(X)|^2 \le \Phi(|X|^2)$ with equality if and only if $\Phi(X) = X$.
}
\end{lemma}
\begin{proof}
Clearly $0 \le (\Phi(X) - X)^*(\Phi(X) - X) = \Phi(X)^*\Phi(X) - X^* \Phi(X) - \Phi(X)^*X + X^*X.$ As $\Phi(X), \Phi(X)^* \in \mathscr{N}$ and $\Phi$ is a faithful positive map (being a $\tau$-preserving conditional expectation), we have $$0 \le \Phi \big( \Phi(X)^*\Phi(X) - X^* \Phi(X) - \Phi(X)^*X + X^*X \big) = \Phi(X^*X) - \Phi(X)^* \Phi(X),$$ with equality if and only if $(\Phi(X)- X)^* (\Phi(X) - X) = 0 \Longleftrightarrow \Phi(X) = X.$
\end{proof}

\begin{lemma}
\label{lem:trace_ineq}
\textsl{
For $X, Y$ in $\mathscr{M}$ such that $0 \le X \le Y$ and a strictly increasing continuous function $f$ on $\mathbb{R}_{+}$, we have $\tau(f(X)) \le \tau(f(Y))$ with equality if and only if $X = Y$.
}
\end{lemma}
\begin{proof}
As $f$ is increasing, using Lemma \ref{lem:s_number},(iv), we have $f(\mu_t(X)) \le f(\mu_t (Y))$ for $t \in [0, 1]$. Using Lemma \ref{lem:s_number}, (vi), we have $\tau(f(X)) = \int_{0}^1 f(\mu_s(X)) \; \textrm{d}s \le \int_{0}^1 f(\mu_s(Y)) \; dy = \tau(f(Y))$.

Let us assume that $\tau(f(X)) = \tau(f(Y))$. Using the right-continuity of the maps $t \in [0, 1]   \mapsto \mu_t(X)$ and $\mu_t(Y)$, we conclude that $f(\mu_t(X)) = f(\mu_t(Y))$ for all $t \in [0, 1]$. As $f$ is strictly increasing, it is a one-to-one function and hence $\mu_t(X) = \mu_t(Y)$ for all $t \in [0, 1]$. Thus $\tau(X) = \int_{0}^1 \mu_s(X) \; \textrm{d}s = \int_{0}^1 \mu_s(Y) \; \textrm{d}s = \tau(Y)$ and by the faithfulness of $\tau$, we have $\tau(Y - X) = 0 \Rightarrow X = Y$. If $X = Y$, equality holds trivially.
\end{proof}

\begin{lemma}[Hardy-Littlewood-P\'{o}lya]
\label{lem:hlp_maj}
\textsl{
Let $\varphi, \psi : [0, 1] \to \mathbb{R}$ be decreasing functions such that $$\int_{0}^t \varphi(s) \; \textrm{d}s \le \int_{0}^t \psi(s) \; \textrm{d}s \textrm{ for } t \in [0, 1), \textrm{ and } \int_{0}^1 \varphi(s) \; \textrm{d}s = \int_{0}^1 \psi(s) \; \textrm{d}s.$$
Then for a continuous convex function $f : \mathbb{R} \to \mathbb{R}$, we have 
\begin{equation}
\label{eqn:hlp_ineq}
\int_{0}^1 f(\varphi(s)) \; \textrm{d}s \le \int_{0}^1 f(\psi(s)) \; \textrm{d}s.
\end{equation}
If $f$ is strictly convex, equality holds in (\ref{eqn:hlp_ineq}) if and only if $\varphi = \psi$ almost everywhere.
}
\end{lemma}

\begin{lemma}[{\cite[Lemma 2]{labuschagne}}]
\label{lem:sqrt_labuschagne}
\textsl{
Consider an invertible positive operator $H \in \mathscr{M}$ and inductively define $H_1 := H$ and $H_{n+1} := \frac{1}{2}(H_n + H H_n^{-1})$ for $n \in \mathbb{N}$. Then $(H_n)_{n \in \mathbb{N} }$ is a decreasing sequence of invertible positive operators in $\mathscr{M}$ converging uniformly to $\sqrt{H}$. 
}
\end{lemma}

\begin{lemma}[Arveson-Beurling factorization theorem]
\label{lem:arveson_fac}
\textsl{
\begin{itemize}
\item[(i)] (\cite[Theorem 4.2.1]{arveson_analyticity}) Every invertible operator in $\mathscr{M}$ admits a factorization $UA$, where $U$ is a unitary operator in $\mathscr{M}$ and $A \in \mathfrak{A} \cap \mathfrak{A}^{-1}$. In particular, an operator $A \in \mathfrak{A}$ which is invertible in $\mathscr{M}$ admits a factorization $U \tilde{A}$ where $U \in \mathfrak{A}$ and $\tilde{A} \in \mathfrak{A} \cap \mathfrak{A}^{-1}$.
\item[(ii)] (\cite[Corollary 4.2.4(ii)]{arveson_analyticity}) Every invertible positive operator in $\mathscr{M}$ is of the form $A^*A$ for some invertible operator $A \in \mathfrak{A} \cap \mathfrak{A}^{-1}.$
\end{itemize}
}
\end{lemma}

\section{The Main Majorization Inequality}
\label{sec:main}

\begin{prop}
\label{prop:maj_phi}
\textsl{
For a positive operator $H$ in $\mathscr{M}$ and $t \in [0, 1]$, we have $$\Sigma_t(\Phi(H)) \le \Sigma_t(H).$$
}
\end{prop}
\begin{proof}
Assume that $\mathscr{M}, \mathscr{N}$ both have no minimal projections. From \cite[Lemma 4.1]{fack-kosaki}, note that $$\Sigma_t(\Phi(H)) = \sup \{ \tau(\Phi(H)E) : E \textrm{ is a projection in } \mathscr{N} \textrm{ such that } \tau(E) \le t \},$$
$$\Sigma_t(H) = \sup \{ \tau(HE) : E \textrm{ is a projection in } \mathscr{M} \textrm{ such that } \tau(E) \le t \}.$$
For a projection $E \in \mathscr{N}$, we have $\tau(\Phi(H)E) = \tau(\Phi(HE)) = \tau(HE).$ Since $\mathscr{N} \subseteq \mathscr{M}$, we conclude from the above variational description for $\Sigma_t(\cdot)$ that $\Sigma_t(\Phi(H)) \le \Sigma_t(H).$

As remarked after \cite[Lemma 4.1]{fack-kosaki}, note that our original assumption on $\mathscr{M}, \mathscr{N}$ (no minimal projections) is not restrictive. Clearly $\mathscr{M} \otimes L^{\infty}([0, 1]; \textrm{d}s)$, $ \mathscr{N} \otimes L^{\infty}([0, 1]; \textrm{d}s)$ have no minimal projections. Let $\iota : L^{\infty}([0, 1]; \textrm{d}s) \to L^{\infty}([0, 1]; \textrm{d}s)$ be the identity map. In this scenario, we consider the conditional expectation $\Phi \otimes \iota : \mathscr{M} \otimes L^{\infty}([0, 1]; \textrm{d}t) \to \mathscr{N} \otimes L^{\infty}([0, 1]; \textrm{d}t)$ which preserves the trace given by $\tau \otimes (\int_{0}^1 \cdot \; \textrm{d}s)$. By Remark \ref{rmrk:variation}, we have $\mu_t(H) = \mu_t \big(H \otimes I), \mu_t(\Phi(H)) = \mu_t \big(\Phi(H) \otimes I)$,  where the $s$-numbers for $H \otimes I, \Phi(H) \otimes I$ are relative to $\tau \otimes (\int_{0}^1 \cdot \; \textrm{d}s)$.   

\end{proof}

\begin{cor}
\label{cor:jensen_seed}
\textsl{
For an operator $A$ in $\mathfrak{A}$ and $t \in [0, 1]$, we have 
$$\Sigma_t(|\Phi(A)|^2) \le \Sigma_t(|A|^2).$$
}
\end{cor}
\begin{proof}
As $|\Phi(A)|^2 \le \Phi(|A|^2)$ (by the generalized Schwarz inequality), we have 
$$\Sigma_t(|\Phi(A)|^2) \le \Sigma_t(\Phi(|A|^2)) \textrm{ for } t \in [0, 1] \;\;\; (\textrm{by Lemma \ref{lem:s_number},(iv)}).$$
The conclusion follows using Proposition \ref{prop:maj_phi} for the positive operator $|A|^2 \in \mathscr{M}$.
\end{proof}

We remind the reader that when discussing the invertibility of an operator in $\mathfrak{A}$, there are two main ambient algebras under consideration: $\mathfrak{A}$ and $\mathscr{M}$. We say that $A \in \mathfrak{A}$ is invertible if $A$ has an inverse in $\mathscr{M}$. If the inverse is also in $\mathfrak{A}$, we say that $A \in \mathfrak{A} \cap \mathfrak{A}^{-1}$.

\begin{prop}
\label{prop:labuschagne_sqrt}
\textsl{
For an invertible operator $A \in \mathfrak{A},$ and $t \in [0, 1]$, we have 
$$\Sigma_t(|\Phi(A)|^{1/2^n} ) \le \Sigma_t( |A|^{1 / 2^n}) \; \forall n \in \mathbb{N}.$$
}
\end{prop}
\begin{proof}
Consider the family of assertions indexed by $r > 0$, $$P(r) : \Sigma_t(|\Phi(A)|^r) \le \Sigma_t(|A|^r) \textrm{ for all invertible operators }A \in \mathfrak{A}  \textrm{ and } t \in [0, 1].$$ 

For a fixed $r > 0$, let us assume that $P(r)$ is true. Consider an invertible operator $A \in \mathfrak{A}$. Let $H_1 = |A|^r$ and inductively define $H_{n+1} := \frac{1}{2}(H_n + |A|^r H_n ^{-1})$ for $n \in \mathbb{N}$. By the Arveson-Beurling factorization theorem (Lemma \ref{lem:arveson_fac}, (ii)), we may choose a sequence of invertible operators $(B_n)_{n \in \mathbb{N} } \subset \mathfrak{A} \cap \mathfrak{A}^{-1}$ such that $ |B_n| = H_n^{\frac{1}{r}}$ for every $n \in \mathbb{N}$. Note that $|(B_n ^{-1})^*| = |B_n|^{-1} = |H_n|^{-\frac{1}{r}} $. By Lemma \ref{lem:cum_sum}, (ii), we observe that $\Sigma_t(|AB_n^{-1}|^r) = \Sigma_t(\big||A| |(B_n^{-1})^*|\big|^r) =  \Sigma_t(\big||A||B_n|^{-1}\big|^r) = \Sigma_t(|A|^r H_n^{-1})$ (the last equality holds because $H_n$ commutes with $|A|$.) For $t \in [0, 1]$, we have
\begin{align*}
&\phantom{=}\frac{1}{2} (\Sigma_t(H_n) + \Sigma_t(|A|^r H_n ^{-1})) \\
& = \frac{1}{2} (\Sigma_t(|B_n|^r) + \Sigma_t(|A B_n^{-1}|^r) )& \\
& \ge \frac{1}{2} (\Sigma_t(|\Phi(B_n)|^r) + \Sigma_t(|\Phi(A B_n^{-1})|^r) ) & (\textrm{by the hypothesis }P(r))\\
& \ge \Sigma_t(|\Phi(A B_n^{-1}) \Phi(B_n)|^{\frac{r}{2}}) & (\textrm{by Lemma } \ref{lem:cum_sum},\textrm{(iii)}) \\
& = \Sigma_t(|\Phi(A)|^{\frac{r}{2}}) & (\textrm{since } \Phi(X)\Phi(Y) = \Phi(XY), \textrm{ for } X, Y \in \mathfrak{A}.)
\end{align*}

By Lemma \ref{lem:sqrt_labuschagne}, $(H_n)_{n\in \mathbb{N}}$ is a decreasing sequence of positive operators uniformly converging to $|A|^{\frac{r}{2}}$. Thus $(|A|^r H_n^{-1})_{n \in \mathbb{N}}$ is an increasing sequence of positive operators uniformly converging to $|A|^{\frac{r}{2}}.$ By Lemma \ref{lem:mon_conv}, $$\lim_{n \rightarrow \infty} \Sigma_t(H_n) = \Sigma_t(|A|^{\frac{r}{2}}) = \lim_{n \rightarrow \infty} \Sigma_t(|A|^{r}H_n^{-1}) \textrm{ for } t \in [0, 1]. $$

Hence $\Sigma_t(|\Phi(A)|^{\frac{r}{2}}) \le \Sigma_t(|A|^{\frac{r}{2}})$ for all invertible operators $A \in \mathfrak{A}$ and $t \in [0, 1]$. Thus $P(r) \Rightarrow P(\frac{r}{2})$. As the assertion $P(2)$ is true (by Corollary \ref{cor:jensen_seed}), we conclude that $P(2^{-n})$ is true for all $n \in \mathbb{N}$.
\end{proof}

\begin{thm}
\label{thm:jensen_subdiagonal}
\textsl{
For an invertible operator $A \in \mathfrak{A}$, we have
$$\int_{0}^t \log \mu_s(|\Phi(A)|) \; \textrm{d}s \le \int_{0}^t \log \mu_s(|A|) \; \textrm{d}s \textrm{ for } t \in [0, 1].$$
In addition, if $A \in \mathfrak{A} \cap \mathfrak{A}^{-1}$, we have
$$\int_{0}^1 \log \mu_s(|\Phi(A)|) \; \textrm{d}s = \int_{0}^1 \log \mu_s(|A|) \; \textrm{d}s.$$
}
\end{thm}
\begin{proof}
For $\lambda \ge 1$ and $n \in \mathbb{N}$, note that $$2^n(\lambda ^{1/2^n } - 1) = 2^{n+1}(\lambda ^{1/2^{n+1}} - 1) \cdot \frac{1}{2}(\lambda ^{1/2^{n+1}} + 1) \ge 2^{n+1}(\lambda ^{1/2^{n+1}} - 1).$$
Thus the sequence $(2^n(\lambda ^{1/2^n} - 1))_{n \in \mathbb{N} } \subset \mathbb{R}_{+}$ is decreasing and converges to $\log \lambda$ (as $\lim_{r \rightarrow 0} \frac{\lambda ^r - 1}{r} = \log \lambda$).

Let $H$ be a positive operator in $\mathscr{M}$ such that $I \le H$. Then $1 \le \mu_t(H)$ for $t \in [0, 1]$. The sequence of functions $t \in [0, 1] \mapsto 2^n(\mu_t(x)^{1/2^n} - 1), n \in \mathbb{N}$ is decreasing and converges pointwise to the function $t \in [0, 1] \mapsto \log \mu_t(x)$. By the monotone convergence theorem, we have $$\lim_{n \rightarrow \infty} \int_{0}^t 2^n(\mu_s(H)^{1/2^n} - 1) \; \textrm{d}s = \int_{0}^t \log \mu_s(H) \; \textrm{d}s, \textrm{ for } t \in [0, 1].$$ 
Without loss of generality, we may assume $I \le |A|,$ and $I \le  |\Phi(A)|$ by appropriately scaling $A$ if necessary. Using Proposition \ref{prop:labuschagne_sqrt}, we conclude that $$\int_{0}^t \log \mu_s(|\Phi(A)|) \; \textrm{d}s \le \int_{0}^t \log \mu_s(|A|) \; \textrm{d}s \textrm{ for } t \in [0, 1].$$
If $A \in \mathfrak{A} \cap \mathfrak{A}^{-1}$, by Remark \ref{rmrk:log_tr} and \cite[Theorem 4.4.3]{arveson_analyticity}, we have 
$$ \int_{0}^1 \log \mu_s(|\Phi(A)|) \; \textrm{d}s =  \tau(\log |\Phi(A)| ) = \tau(\log |A|) = \int_{0}^1 \log \mu_s(|A|) \; \textrm{d}s.$$
\end{proof}

\section{Applications}
\label{sec:app}

\begin{thm}
\label{thm:main_ineq}
\textsl{
Let $f : \mathbb{R}_{+} \to \mathbb{R}$ be an increasing continuous function such that $f \circ \exp$ is convex on $\mathbb{R}$. For $A \in \mathfrak{A}$, we have 
\begin{equation}
\label{eqn:main_ineq}
\tau(f(|\Phi(A)|)) \le \tau(f(|A|)).\end{equation}
In addition, if $A$ is invertible in $\mathscr{M}$ and $f \circ \exp$ is strictly convex, equality holds in (\ref{eqn:main_ineq}) if and only if $\Phi(A) = A$.
}
\end{thm}

\begin{proof}
We first prove the result for operators in $\mathfrak{A} \cap \mathfrak{A}^{-1}$. Let $A \in \mathfrak{A} \cap \mathfrak{A}^{-1}$. Applying the Hardy-Littlewood-P\'{o}lya inequality (Lemma \ref{lem:hlp_maj}) in the context of Theorem \ref{thm:jensen_subdiagonal}, we get inequality (\ref{eqn:main_ineq}). Suppose that $f \circ \exp$ is strictly convex and $\tau(f(|\Phi(A)|)) = \tau(f(|A|))$. By the equality condition in Lemma \ref{lem:hlp_maj} and right-continuity of $t \mapsto \mu_t(\cdot)$, we conclude that $\mu_t(|\Phi(A)|) = \mu_t(|A|)$ for $t \in [0, 1]$. Thus $\tau(|\Phi(A)|^2) = \int_{0}^1 \mu_s(|\Phi(A)|)^2 \; \textrm{d}s = \int_{0}^1 \mu_s(|A|)^2 \; \textrm{d}s = \tau(|A|^2) = \tau(\Phi(|A|^2)) \Rightarrow |\Phi(A)|^2 = \Phi(|A|^2) \Rightarrow \Phi(A) = A$ (by Lemma \ref{lem:schwarz_ineq}).

We next prove the inequality under the weaker hypothesis that $A \in \mathfrak{A}$ is invertible in $\mathscr{M}$. By Lemma \ref{lem:arveson_fac}, (i), there is a unitary $U$ in $\mathfrak{A}$ and $\tilde{A} \in \mathfrak{A} \cap \mathfrak{A}^{-1}$ such that $A = U \tilde{A}$. Using the generalized Schwarz inequality (Lemma \ref{lem:schwarz_ineq}), we note that $|\Phi(A)|^2 = \Phi(\tilde{A})^*\Phi(U)^*\Phi(U) \Phi(\tilde{A}) \le \Phi(\tilde{A})^* \Phi(U^*U) \Phi(\tilde{A}) = |\Phi(\tilde{A})|^2$. By the operator monotonicity of the map $t \in \mathbb{R}_{+} \mapsto \sqrt{t}$, we have $|\Phi(A)| \le |\Phi(\tilde{A})|$, with equality if and only if $\Phi(U) = U$ (by Lemma \ref{lem:schwarz_ineq} and as $\Phi(\tilde{A})$ is invertible). Since $f$ is increasing and inequality (\ref{eqn:main_ineq}) holds for $\tilde{A}$, using Lemma \ref{lem:trace_ineq} we have $\tau(f(|\Phi(A)|)) \le \tau(f(|\Phi(\tilde{A})|)) \le \tau(f(|\tilde{A}|)) = \tau(f(|A|))$ which proves inequality (\ref{eqn:main_ineq}) for $A$.  Suppose that $f \circ \exp$ is strictly convex and $\tau(f(|\Phi(A)|)) =  \tau(f(|A|))$. Then $f$ is strictly increasing and we have $\Phi(U) = U$ and $\Phi(\tilde{A}) = \tilde{A}$. Thus if equality holds in (\ref{eqn:main_ineq}), we have $\Phi(A) = \Phi(U \tilde{A}) = \Phi(U) \Phi(\tilde{A}) = U \tilde{A} = A$.

The only thing that remains to be proved is inequality (\ref{eqn:main_ineq}) when $A$ is not invertible. Let $\varepsilon > 0$. By Arveson's factorization theorem (Lemma \ref{lem:arveson_fac}, (ii)), there is an invertible operator $B$ in $\mathfrak{A} \cap \mathfrak{A}^{-1}$ such that $\varepsilon I + A^*A = B^*B$. We have $\varepsilon (B^{-1})^* B^{-1} +  (A B^{-1})^*(A B^{-1}) = I$. Using the generalized Schwarz inequality (Lemma \ref{lem:schwarz_ineq}), we have $I = \Phi(I) \ge \varepsilon \Phi(B^{-1})^* \Phi(B^{-1}) +  \Phi(A B^{-1})^*\Phi(A B^{-1}) = \varepsilon \Phi(B^{-1})^* \Phi(B^{-1}) + \Phi(B^{-1})^* \Phi(A)^* \Phi(A) \Phi(B^{-1})$ (since $\Phi(A B^{-1}) = \Phi(A) \Phi(B)^{-1}$). Thus $|\Phi(B)|^2 = \Phi(B)^*\Phi(B) \ge \varepsilon I + \Phi(A)^* \Phi(A) = \varepsilon I + |\Phi(A)|^2 \ge |\Phi(A)|^2$. Using the operator monotonicity of the map  $t \in \mathbb{R}_{+} \mapsto \sqrt{t}$, note that $|\Phi(A)|  \le |\Phi(B)|.$ As $f$ is increasing and inequality (\ref{eqn:main_ineq}) holds for $B$, we have $\tau(f(|\Phi(A)|)) < \tau(f(|\Phi(B)|)) \le \tau(f(|B|)) = \tau(f(\sqrt{\varepsilon I + |A|^2})).$ Taking the limit as $\varepsilon \rightarrow 0$, we conclude that $\tau(f(|\Phi(A)|)) \le \tau(f(|A|))$.

\end{proof}

\begin{remark}
Let $\mathcal{F}$ denote the set of increasing continuous functions  $f : \mathbb{R}_{+}  \to \mathbb{R}_{+}$ such that $f \circ \exp$ is convex on $\mathbb{R}$. Let $g : \mathbb{R}_{+} \to \mathbb{R}_{+}$ be an increasing convex function. Then 
\begin{itemize}
\item[(i)] $g \in \mathcal{F},$
\item[(ii)] if $f \in \mathcal{F},$ then $g \circ f \in \mathcal{F}$,
\item[(iii)] for $r > 0$ and $f \in \mathcal{F}$, the function $t \in  \mathbb{R}_{+} \mapsto f(t^r)$ belongs to $\mathcal{F}$.
\end{itemize}
Examples of functions in $\mathcal{F}$ include (for $r > 0$) $e^{rt}, t^r, \log(1+t^r)$, etc. This remark serves to illustrate the applicability of Theorem \ref{thm:main_ineq} for a rich class of commonly used functions.

\end{remark}

\begin{cor}
\label{cor:app}
\textsl{
For $A \in \mathfrak{A}$ and $r > 0$, we have 
\begin{itemize}
\item[(i)] $\tau(|\Phi(A)|^r) \le \tau(|A|^r)$,
\item[(ii)] $\Delta(I + |\Phi(A)|^r) \le \Delta(I + |A|^r)$.
\end{itemize}
If $A$ is invertible in $\mathscr{M}$, equality holds in either of the above two inequalities if and only if $\Phi(A) = A$.
}
\end{cor}
\begin{proof}
The functions $t \in \mathbb{R}_{+} \mapsto t^r, t \in \mathbb{R}_{+} \mapsto \log(1+t^r)$ are both increasing. Note that as $$\frac{\textrm{d}^2}{\textrm{d}t^2}e^{rt} = r^2 e^{rt} > 0, \frac{\textrm{d}^2}{\textrm{d}t^2} \log(1+e^{rt}) = \frac{r^2 e^{rt}}{(1+e^{rt})^2} > 0,$$ the functions $t \in \mathbb{R} \mapsto e^{rt}, t \in \mathbb{R}  \mapsto \log(1 + e^{rt})$ are strictly convex. Thus the result follows from Theorem \ref{thm:main_ineq}.
\end{proof}

\begin{remark}
For $A \in M_n(\mathbb{C})$, we have $\Delta(A) = |\det (A) | ^{\frac{1}{n}}.$ Thus Corollary \ref{cor:app},(ii), generalizes Theorem \ref{thm:lin} and Lemma \ref{lem:drury} when considered in the context of the finite subdiagonal algebras described in Example \ref{ex:upper_tri}.
\end{remark}

Let $A$ be an operator in a finite subdiagonal algebra $\mathfrak{A}$. As $\Phi$ is a contractive map (being a conditional expectation), we observe that $\| \Phi(A)^n \| = \| \Phi(A^n) \| \le \|A^n \|.$ Thus we have $$r(\Phi(A)) = \lim_{n \rightarrow \infty} \|\Phi(A)^n \|^{\frac{1}{n}} \le \lim_{n \rightarrow \infty} \|A^n \|^{\frac{1}{n}} = r(A).$$
(Note that the above inequality is valid in any subdiagonal algebra, not just finite subdiagonal algebras.)  By Corollary \ref{cor:app}, (i), we have $\tau(|\Phi(A)|^r) \le \tau(|A|^r)$ for all $r > 0$. As $\tau$ is normal and $(|A|^r)_{0 < r \le 1}$ is a bounded family of positive operators converging in the strong-operator topology to $R(|A|)$ (the range projection of $|A|$) as $r \rightarrow 0$, we have 
\begin{equation}
\label{eqn:range_ineq}
\tau(R(|\Phi(A)|)) \le \tau(R(|A|)).
\end{equation}
Let $\lambda \in \sigma_p(A)$ so that $N(A - \lambda I) \ne 0$. Note that $N(X) = N(X^*X) = N(|X|) = I - R(|X|)$ for all $X \in \mathscr{M}$. Thus using inequality (\ref{eqn:range_ineq}), we have $0 < \tau(N(A-\lambda I)) = \tau(I - R(|A - \lambda I|)) = 1 - \tau(R(|A - \lambda I|)) \le 1 - \tau(R(|\Phi(A - \lambda I)|)) = \tau(I - R(|\Phi(A) - \lambda I|)) = \tau(N(\Phi(A) - \lambda I))$ which shows that $N(\Phi(A) - \lambda I) \ne 0$. Hence $\lambda \in \sigma_p(\Phi(A)).$ We summarize the above discussion in the form of a theorem below.

\begin{thm}
\label{thm:point_spec}
\textsl{
For $A \in \mathfrak{A}$, we have
\begin{itemize}
\item[(i)] $r(\Phi(A)) \le r(A)$,
\item[(ii)] $\sigma_p(A) \subseteq \sigma_p(\Phi(A)).$
\end{itemize}
}
\end{thm}

A natural question that comes to mind is whether we can say something stronger about the relationship between $\sigma(A)$ and $\sigma(\Phi(A))$. For instance, in the context of Example \ref{ex:upper_tri}, the spectrum of a block upper triangular matrix is identical to the spectrum of its diagonal. One may wonder whether that is always the case for operators in finite subdiagonal algebras. It turns out that neither the containment relation $\sigma(\Phi(A)) \subseteq \sigma(A)$ nor $\sigma(A) \subseteq \sigma(\Phi(A))$ holds in general. In the context of Example \ref{ex:hardy}, (i), consider $A$ to be the coordinate function $z \in H^{\infty}(\mathbb{T})$. Note that $\sigma(\Phi(A)) = \{ 0 \}$ (as $\Phi(A) = (\int_{\mathbb{T}} z \; \mathrm{d}\mu)I = 0)$, $\sigma(A) = \mathbb{T}$ and clearly  $\{ 0 \} \not\subset \mathbb{T}, \mathbb{T} \not\subset \{0 \}.$  Further in this scenario as $\sigma_p(A) = \varnothing \subset \{ 0 \} = \sigma_p(\Phi(A))$, we observe that it is possible to have a strict containment relation in Theorem \ref{thm:point_spec}, (ii).

\end{document}